\documentclass[11pt,reqno]{amsart}

\usepackage[margin=1.15in]{geometry}
\usepackage{amsmath,amssymb,mathtools}
\usepackage{microtype}
\usepackage{xcolor}
\usepackage[colorlinks=true,
  linkcolor=blue!45!black,
  citecolor=blue!45!black,
  urlcolor=blue!45!black,
  pdftitle={Metrics with positive Ricci eigenvalue sums on torus-sphere products},
  pdfauthor={Huibin Chen and Yiyue Zhang}]{hyperref}

\numberwithin{equation}{section}
\newtheorem{theorem}{Theorem}[section]
\newtheorem{proposition}[theorem]{Proposition}
\newtheorem{lemma}[theorem]{Lemma}
\newtheorem{corollary}[theorem]{Corollary}
\theoremstyle{remark}
\newtheorem{remark}[theorem]{Remark}

\newcommand{\TT}{\mathbb T}
\newcommand{\SSph}{\mathbb S}
\newcommand{\RR}{\mathbb R}
\DeclareMathOperator{\Ric}{Ric}
\DeclareMathOperator{\Rm}{Rm}
\DeclareMathOperator{\tr}{tr}
\DeclareMathOperator{\spec}{spec}
\DeclareMathOperator{\dist}{dist}
\DeclareMathOperator{\Id}{Id}
\DeclareMathOperator{\Sym}{Sym}
\DeclareMathOperator{\vol}{vol}
\DeclareMathOperator{\diag}{diag}
\newcommand{\RicSum}[1]{\mathcal S_{#1}}
\newcommand{\Bcal}{\mathcal B}
\newcommand{\Mcal}{\mathcal M}
\DeclareMathOperator{\fintavg}{Avg}

\title[Metrics with positive Ricci eigenvalue sums on $\mathbb{T}^k\times \mathbb{S}^m$]
{Metrics with positive Ricci eigenvalue sums on $\mathbb{T}^k\times \mathbb{S}^m$}

\author{Huibin Chen}
\address{School of Mathematical Sciences, Nanjing Normal University, Nanjing 210023, China}
\email{chenhuibin@njnu.edu.cn}
\author{Yiyue Zhang}
\address{
Beijing Institute of Mathematical Sciences and Applications,
Beijing 101408, China
}
\email{zhangyiyue@bimsa.cn}
\date{}

\subjclass[2020]{Primary 53C21; Secondary 53C20, 58J37}
\keywords{$k$-positive Ricci curvature, torus--sphere product,
Kaluza--Klein metric, Ricci eigenvalue sums, spectral perturbation,
harmonic one-form, Bochner formula}

\begin{document}

\begin{abstract}
We construct metrics with $k$-positive Ricci curvature on
$\TT^k\times\SSph^m$, for $k,m\geq2$, arbitrarily close to the
standard flat-round product in the smooth topology. The two-torus
case disproves Wolfson's conjecture that every closed manifold with
two-positive Ricci curvature has virtually free fundamental group.
We also establish a local obstruction to two-positive Ricci curvature
near closed backgrounds carrying three orthonormal parallel one-forms.
\end{abstract}

\maketitle

\section{Introduction}

Let $(M^n,g)$ be a Riemannian manifold. At $p\in M$, write
\[
  \lambda_1(\Ric_g^\sharp(p))\leq\cdots\leq\lambda_n(\Ric_g^\sharp(p))
\]
for the eigenvalues of the Ricci endomorphism. For $1\leq k\leq n$, set
\[
  \RicSum{k}(g,p)=\sum_{\ell=1}^k\lambda_\ell(\Ric_g^\sharp(p)).
\]
Following Wolfson~\cite{Wo2011}, we say that $g$ has \emph{$k$-positive Ricci curvature} if $\RicSum{k}(g,p)>0$ at every point $p\in M$. Ky Fan's minimum principle gives
\begin{equation}
 \RicSum{k}(g,p)=
 \min_{\substack{u_1,\ldots,u_k\in T_pM\\
 |u_\ell|=1,\ \langle u_\ell,u_q\rangle=0\ (\ell\neq q)}}
 \sum_{\ell=1}^k\Ric_g(u_\ell,u_\ell).
 \label{eq:ky-fan}
\end{equation}
Thus $\RicSum{k}>0$ is equivalent to positivity of the Ricci trace on every $k$-plane. Lower bounds on $2$-Ricci curvature are used
in~\cite{HiKaKhZh2025,Zh2021} to obtain a sharper bound on the injectivity
radius of the Clifford torus in $\mathbb{S}^3$. This condition differs in general from intermediate Ricci curvature,
bi-Ricci curvature, and positivity conditions on the curvature operator.
For bi-Ricci curvature, see Shen and Ye~\cite{ShYe1996,ShYe1997};
for the intermediate curvature conditions of Brendle, Hirsch, and Johne,
see~\cite{BrHiJo2024}. 

Recently, Chen and Zhang~\cite{ChZh2026} constructed metrics with positive
biorthogonal curvature on $\SSph^2\times\TT^2$ arbitrarily close to
the standard product in the smooth topology. Their result concerns a
different curvature condition but likewise shows that the fundamental
group need not be virtually free.

For $k\geq1$, let $\TT^k=(\RR/2\pi\mathbb Z)^k$ carry its standard flat metric, and let $g_0^{(k)}$ be the flat--round product metric on $\TT^k\times\SSph^m$. Its Ricci spectrum consists of $k$ zero eigenvalues and $m$ eigenvalues equal to $m-1$. In particular, $\RicSum{k}(g_0^{(k)})=0$. The first main result perturbs this equality to strict positivity.

\begin{theorem}
\label{thm:construction}
Let $m\geq2$ and $k\geq2$. There are $s_0>0$ and a smooth family of metrics $g_s^{(k)}$ on $\TT^k\times\SSph^m$, defined for $0\leq s<s_0$, with $g_s^{(k)}|_{s=0}=g_0^{(k)}$ and $g_s^{(k)}\to g_0^{(k)}$ in $C^\infty$ as $s\to0$, such that, uniformly in $p\in\TT^k\times\SSph^m$,
\begin{equation}\label{eq:construction-asymptotic}
 \RicSum{k}(g_s^{(k)},p)
 =\frac{s^6}{4(m+1)}+O(s^8).
\end{equation}
In particular, $\RicSum{k}(g_s^{(k)})>0$ for every sufficiently small $s>0$. The construction is active in two torus directions; the remaining $k-2$ directions form a flat product factor.
\end{theorem}
Within the family $\TT^k\times\SSph^{n-k}$, with $1\leq k\leq n-1$,
the two excluded endpoints have topological obstructions. When $k=1$,
the condition $\RicSum{1}>0$ is positive Ricci curvature, which is
incompatible with the infinite fundamental group by the Bonnet--Myers
theorem. When $k=n-1$, the product is $\TT^n$. Positivity of
$\RicSum{n-1}$ forces $\lambda_{n-1}>0$, hence $\lambda_n>0$ and
positive scalar curvature. This is impossible on a torus~\cite{GrLa1980}.

Cucinotta and Mondino~\cite[Theorem~1.6]{CuMo2026} prove first
Betti number bounds under almost nonnegative Ricci curvature,
diameter and noncollapsing bounds, and a fixed positive lower bound
for normalized local curvature integrals. The allowed negative
Ricci tolerance depends on this positive lower bound. In our family,
$\RicSum{k}$ converges uniformly to zero, so the relevant integral
averages tend to zero; the fixed positive integral hypothesis
therefore does not hold uniformly.

The case $k=2$ has an immediate topological consequence. Wolfson~\cite[Conjecture~1]{Wo2011} conjectured that the fundamental group of a closed manifold with two-positive Ricci curvature must be virtually free. Ramachandran and Wolfson~\cite{RaWo2010} proved this conclusion under a bounded fill-radius assumption on the universal cover, while the proposed implication from curvature to fill radius remained conjectural.

\begin{corollary}
For every $m\geq2$, the closed manifold $\TT^2\times\SSph^m$ admits two-positive Ricci metrics arbitrarily $C^\infty$-close to the flat--round product. Consequently, two-positive Ricci curvature does not force the fundamental group of a closed manifold to be virtually free.
\end{corollary}

\begin{proof}
Apply Theorem~\ref{thm:construction} with $k=2$. Every finite-index subgroup of $\mathbb Z^2$ is isomorphic to $\mathbb Z^2$, which is not a free group. Hence $\pi_1(\TT^2\times\SSph^m)=\mathbb Z^2$ is not virtually free.
\end{proof}

For two-positive Ricci curvature, the local behavior changes when a
third parallel direction is present. The complementary result below
applies to any closed connected background carrying three orthonormal
parallel one-forms. In particular, every sufficiently small
$C^{2,\alpha}$ perturbation of the standard flat--round product metric
on $\TT^3\times\SSph^m$ fails to have two-positive Ricci curvature.

\begin{theorem}\label{thm:obstruction-intro}
Let $(M,g_0)$ be a closed connected Riemannian manifold admitting three pointwise $g_0$-orthonormal parallel one-forms. For every $0<\alpha<1$, there is a $C^{2,\alpha}$-neighborhood $\mathcal U$ of $g_0$ in the space of smooth metrics such that no $g\in\mathcal U$ satisfies $\RicSum{2}(g)>0$ everywhere.
\end{theorem}

We now outline the construction of the metrics in
Theorem~\ref{thm:construction}.
We first work on $\TT^2\times\SSph^m$. We deform the metric on each
torus fiber, with parameter $s$, while keeping its volume fixed.
At this stage, the sum of the two small Ricci eigenvalues remains
zero. We then introduce a connection perturbation of strength $t$,
which couples the torus and sphere directions. The coefficient of
$t^2$ in this eigenvalue sum has positive average over the sphere,
but need not be positive everywhere. To obtain pointwise positivity,
we adjust the fiber volume by solving a Poisson equation. This
replaces the coefficient by its positive average. Choosing $t=s^2$
makes the remaining error smaller than the positive leading term.
Finally, taking the product with a flat $(k-2)$-torus gives the
construction on $\TT^k\times\SSph^m$.

The obstruction uses harmonic representatives of three parallel
cohomology classes. For nearby metrics these remain close to an
orthonormal frame. A three-dimensional trace estimate then makes the
sum of their Ricci terms strictly positive, contradicting the
integrated Bochner formula.

Sections~\ref{sec:kk} and~\ref{sec:cluster} establish the Ricci and
spectral perturbation formulas. Sections~\ref{sec:spherical}
and~\ref{sec:correction} construct the metrics and prove
Theorem~\ref{thm:construction}; Appendix~\ref{app:algebra} supplies
the contraction calculations used in that proof.
Section~\ref{sec:obstruction} proves
Theorem~\ref{thm:obstruction-intro}.

\medskip

\noindent
\textbf{AI usage:} The authors used ChatGPT extensively in developing this manuscript and generating the main counterexamples. We independently checked and rewrote the proofs and take full responsibility for the content.

\medskip

\noindent
\textbf{Acknowledgements:}
YZ thanks Professor Schoen for many helpful
discussions and suggestions.  YZ was partially supported by NSFC Grant
No.\ 12501070 and by the startup fund from BIMSA. HB was partially supported by NSFC Grant
No.\ 12571025.

\section{Torus-invariant Ricci formulas}\label{sec:kk}

We use the Einstein summation convention, with $i,j,\ell,\ldots$ denoting
base indices, raised using $h$, and $a,b,c,\ldots$ denoting fiber
indices, raised using $H$ when they are tensor indices. The
coefficient-matrix labels introduced below are an exception to this
convention. Our curvature and Laplacian conventions are
\begin{align}
\Rm(X,Y)Z
&=\nabla_X\nabla_YZ-\nabla_Y\nabla_XZ-\nabla_{[X,Y]}Z,
\label{eq:curvature-convention}\\
\Ric(Y,Z)
&=\sum_\gamma
\bigl\langle \Rm(e_\gamma,Y)Z,e_\gamma\bigr\rangle,
\qquad
\Delta=\nabla^i\nabla_i,
\notag
\end{align}
where $\{e_\gamma\}$ is a local orthonormal frame.
With these conventions, the unit sphere satisfies
\[
\Ric_{\SSph^m}=(m-1)g_{\SSph^m},
\qquad
\Delta x_\alpha=-mx_\alpha
\]
for each ambient coordinate function restricted to $\SSph^m$.
We normalize the inner product on two-forms by
\begin{equation}
\langle P,Q\rangle_{\Lambda^2}
=\frac12P_{ij}Q^{ij}.
\label{eq:two-form-convention}
\end{equation}
For columns of two-forms, the unweighted pairing is the componentwise
Euclidean pairing:
\[
  \langle P,Q\rangle:=\sum_{a=1}^k
  \langle P^a,Q^a\rangle_{\Lambda^2},
  \qquad |P|^2:=\langle P,P\rangle.
\]
Let $(B^m,h)$ be closed. Fix a smooth map $H=(H_{ab}):B\to\Sym^+(k)$ and an $\RR^k$-valued one-form $\alpha=(\alpha^1,\ldots,\alpha^k)^T$. Set $F^a=d\alpha^a$ and consider the torus-invariant metric
\begin{equation}
 g_t=h+H_{ab}\eta^a\eta^b,
 \qquad
 \eta^a=d\theta^a+t\alpha^a.
 \label{eq:kk-metric}
\end{equation}
Here $H_{ab}\eta^a\eta^b$ denotes the corresponding symmetric two-tensor. Define
\begin{equation}
 \nu=(\det H)^{1/2},
 \qquad
 L_i=\frac12H^{-1}\nabla_iH,
 \qquad
 (L_i)^b{}_a=\frac12H^{bc}\nabla_iH_{ca}.
 \label{eq:nu-L-definitions}
\end{equation}
The function $\nu$ is the fiber-volume density. Taking the trace in \eqref{eq:nu-L-definitions} and using $\nabla_i\log\det H=\tr(H^{-1}\nabla_iH)$, we obtain
\begin{equation}
 \tr L_i=\nabla_i\log\nu.
 \label{eq:trace-L}
\end{equation}

For a column $G=(G^1,\ldots,G^k)^T$ of two-forms, set
\[
  \begin{aligned}
    \Bcal(G)_{ia}&=\nabla^jG^a_{ij},\\
    \Mcal(G)_{ab}&=G^a_{ij}G^{b\,ij}.
  \end{aligned}
\]
Here the lower labels $a,b$ in $\Bcal(G)_{ia}$ and
$\Mcal(G)_{ab}$ denote matrix slots in the fixed torus frame; they
do not lower the fiber indices of $G$ using $H$. The divergence is taken componentwise, and $\Mcal(G)$ records
pairwise contractions of the components. We regard $\Bcal(G)$ as a
map $\RR^k\to TB$ after raising only the base index. With the Hodge convention $(\delta G^a)_i=-\nabla^jG^a_{ji}$, antisymmetry gives $\Bcal(G)_{ia}=(\delta G^a)_i$.

We next choose an adapted frame. Write $U_a=\partial_{\theta^a}$ for the coordinate fields along the torus. Given a vector field $e_i$ on $B$, let
\[
  E_i=e_i-t\alpha^a(e_i)U_a
\]
be its horizontal lift. Then
\[
  \eta^a(E_i)=0, \qquad \eta^a(U_b)=\delta^a_b,
\]
and hence
\[
  \begin{aligned}
    g_t(E_i,E_j)&=h(e_i,e_j),\\
    g_t(E_i,U_a)&=0,\\
    g_t(U_a,U_b)&=H_{ab}.
  \end{aligned}
\]
Thus the two distributions are orthogonal, although $(U_a)$ is not usually an orthonormal vertical frame.

Writing $\alpha_i^a=\alpha^a(e_i)$, we obtain
\begin{align}
 [E_i,E_j]
 &= [e_i,e_j]
    -t\bigl(e_i(\alpha_j^a)-e_j(\alpha_i^a)\bigr)U_a
 \notag\\
 &= [e_i,e_j]^{\mathrm h}-tF^a_{ij}U_a,
 \qquad
 [E_i,U_a]=0.
 \label{eq:adapted-brackets}
\end{align}
Thus $F$ measures the nonintegrability of the horizontal distribution.

We now compute the Levi--Civita connection. Fix $p\in B$ and choose an $h$-orthonormal frame that is normal at $p$, so
\[
  h(e_i,e_j)(p)=\delta_{ij}, \qquad \nabla^h_{e_i}e_j(p)=0.
\]
Torsion-freeness gives $[e_i,e_j](p)=0$. At $p$ we have
\begin{align}
 \nabla^g_{E_i}E_j
 &=-\frac t2F^a_{ij}U_a,
 \label{eq:koszul-horizontal}\\
 \nabla^g_{E_i}U_a=\nabla^g_{U_a}E_i
 &=(L_i)^b{}_aU_b
   +\frac t2H_{ab}F^b_i{}^jE_j,
 \label{eq:koszul-mixed}\\
 \nabla^g_{U_a}U_b
 &=-\frac12(\nabla^iH_{ab})E_i.
 \label{eq:koszul-vertical}
\end{align}
Here $F^b_i{}^j=h^{j\ell}F^b_{i\ell}$.

The Koszul formula and \eqref{eq:adapted-brackets} give
\[
  \begin{aligned}
    2g(\nabla^g_{E_i}E_j,U_a)&=-tH_{ab}F^b_{ij},\\
    g(\nabla^g_{E_i}E_j,E_\ell)&=0
  \end{aligned}
\]
at the normal point. For the mixed derivative, they give
\[
  \begin{aligned}
    2g(\nabla^g_{E_i}U_a,U_b)&=\nabla_iH_{ab},\\
    2g(\nabla^g_{E_i}U_a,E_j)&=tH_{ab}F^b_{ij}.
  \end{aligned}
\]
Finally,
\[
  \begin{aligned}
    2g(\nabla^g_{U_a}U_b,E_i)&=-\nabla_iH_{ab},\\
    g(\nabla^g_{U_a}U_b,U_c)&=0.
  \end{aligned}
\]
Raising the paired indices proves \eqref{eq:koszul-horizontal}--\eqref{eq:koszul-vertical}; torsion-freeness gives the equality in \eqref{eq:koszul-mixed}. Away from the normal point, the first formula also contains $(\nabla^h_{e_i}e_j)^{\mathrm h}$.

We use $(E_i)$ to identify the horizontal space with $\RR^m$ and the isometry $\mathcal I_H(z)=(H^{-1/2}z)^aU_a$ to identify the vertical space with $\RR^k$. All transposes and adjoints below refer to these Euclidean identifications. For $(HF)_a=H_{ab}F^b$, the mixed-block notation means
\[
  \Bcal(\nu HF)_{ia}=\nabla^j(\nu H_{ab}F^b_{ij}),
\]
with the fiber components taken in the fixed torus frame.

The following formulas are the torus-invariant reduction formulas;
see Lott~\cite[\S4.2, equation~(4.6)]{Lo2010}. We include a derivation
to fix the signs, normalizations, and frame identifications used here.

\begin{lemma}\label{lem:kk-blocks}
In the orthogonal horizontal--vertical splitting, the Ricci endomorphism has block form
\[
  \Ric^{\sharp}_{g_t}=\begin{pmatrix}A_t&C_t\\ C_t^*&D_t\end{pmatrix}.
\]
Its horizontal, upper-right mixed, and vertical blocks are
\begin{align}
 (A_t)_{ij}
 &=(\Ric_h)_{ij}
   -(\nabla^2\log\nu)_{ij}
   -\tr(L_iL_j)
   -\frac{t^2}{2}H_{ab}F^a_{i\ell}F^b_j{}^\ell,
 \label{eq:A-block}\\
 C_t
 &=\frac{t}{2\nu}\mathcal B(\nu HF)H^{-1/2},
 \label{eq:C-block}\\
 D_t
 &=H^{1/2}\left[-\nu^{-1}\nabla^i(\nu L_i)\right]H^{-1/2}
   +\frac{t^2}{4}H^{1/2}\mathcal M(F)H^{1/2}.
 \label{eq:D-block}
\end{align}
Here $C_t:\RR^k\to T_pB$, and the lower-left block is its Euclidean adjoint $C_t^*$. In \eqref{eq:D-block}, the covariant derivative is the base Levi--Civita derivative on the one-form index of $L$ and acts entrywise on the matrix coefficients.

In particular,
\begin{equation}
 \tr D_t
 =-\nu^{-1}\Delta_h\nu
  +\frac{t^2}{4}H_{ab}F^a_{ij}F^{b\,ij}.
 \label{eq:vertical-trace}
\end{equation}
Set
\[
  \langle HF,F\rangle:=H_{ab}\langle F^a,F^b\rangle_{\Lambda^2}.
\]
The convention \eqref{eq:two-form-convention} gives
\[
  H_{ab}F^a_{ij}F^{b\,ij}=2\langle HF,F\rangle.
\]
Thus the second term in \eqref{eq:vertical-trace} is $\frac{t^2}{2}\langle HF,F\rangle$.
\end{lemma}

\begin{proof}
Since $(U_a)$ is not orthonormal, curvature contracts in the adapted frame as
\begin{equation}
 \Ric_g(X,Y)
 =\sum_\ell g\bigl(\Rm(E_\ell,X)Y,E_\ell\bigr)
  +H^{cd}g\bigl(\Rm(U_c,X)Y,U_d\bigr).
 \label{eq:adapted-ricci-contraction}
\end{equation}
At the chosen normal point, substitute \eqref{eq:koszul-horizontal}--\eqref{eq:koszul-vertical} into \eqref{eq:curvature-convention} and use \eqref{eq:adapted-ricci-contraction}. Set $S_{ij}=H_{ab}F^a_{i\ell}F^b_j{}^\ell$. The horizontal and vertical contractions give
\begin{align}
 \sum_\ell g\bigl(\Rm(E_\ell,E_i)E_j,E_\ell\bigr)
 &=(\Ric_h)_{ij}-\frac{3t^2}{4}S_{ij},
 \notag\\
 H^{cd}g\bigl(\Rm(U_c,E_i)E_j,U_d\bigr)
 &=-\nabla_i\tr L_j
   -\tr(L_iL_j)+\frac{t^2}{4}S_{ij}.
 \label{eq:horizontal-curvature-contractions}
\end{align}
Using $\tr L_j=\nabla_j\log\nu$, \eqref{eq:horizontal-curvature-contractions} gives \eqref{eq:A-block}.

For the mixed component, we put
\[
  P_{aij}=H_{ab}F^b_{ij}, \qquad Q_{ia}=H_{bc}(L_j)^b{}_aF^c_i{}^j.
\]
The two contractions are
\begin{align}
 \sum_j g\bigl(\Rm(E_j,E_i)U_a,E_j\bigr)
 &=\frac t2\nabla^jP_{aij}+\frac t2Q_{ia},
 \notag\\
 H^{bc}g\bigl(\Rm(U_b,E_i)U_a,U_c\bigr)
 &=\frac t2(\tr L^j)P_{aij}-\frac t2Q_{ia}.
 \label{eq:mixed-curvature-contractions}
\end{align}
The $Q_{ia}$ terms in \eqref{eq:mixed-curvature-contractions} cancel. Using $\tr L_j=\nabla_j\log\nu$, we get
\begin{align}
 \Ric_g(E_i,U_a)
 &=\frac{t}{2\nu}\nabla^j
   \bigl(\nu H_{ab}F^b_{ij}\bigr).
 \label{eq:adapted-mixed-ricci}
\end{align}

For the vertical--vertical component, the contractions add to
\begin{align}
 \Ric_g(U_a,U_b)
 &=-\frac12\Delta_hH_{ab}
   -\frac12\langle d\log\nu,dH_{ab}\rangle_h
   +\frac12H^{cd}
      \langle dH_{ac},dH_{bd}\rangle_h
 \notag\\
 &\qquad
   +\frac{t^2}{4}H_{ac}H_{bd}
      F^c_{ij}F^{d\,ij}.
 \label{eq:adapted-vertical-ricci}
\end{align}

Evaluating \eqref{eq:adapted-mixed-ricci} on $\mathcal I_H(z)$ multiplies the coordinate mixed block on the right by $H^{-1/2}$ and proves \eqref{eq:C-block}.

Let $\Ric^{\mathcal V}$ be the symmetric matrix in \eqref{eq:adapted-vertical-ricci}. The vertical Ricci endomorphism has coordinate matrix $H^{-1}\Ric^{\mathcal V}$. Differentiating $L_i=\frac12H^{-1}\nabla_iH$ gives
\begin{equation}
 H^{-1}\Ric^{\mathcal V}
 =-\nu^{-1}\nabla^i(\nu L_i)+\frac{t^2}{4}\Mcal(F)H.
 \label{eq:vertical-coordinate-matrix}
\end{equation}
Conjugating \eqref{eq:vertical-coordinate-matrix} by $H^{1/2}$ proves \eqref{eq:D-block}. Equivalently, $D_t=H^{-1/2}\Ric^{\mathcal V}H^{-1/2}$, which is symmetric.

By \eqref{eq:trace-L},
\[
  \tr\bigl(\nu^{-1}\nabla^i(\nu L_i)\bigr)=\nu^{-1}\nabla^i(\nu\nabla_i\log\nu)=\nu^{-1}\Delta_h\nu.
\]
Taking the trace in \eqref{eq:D-block} proves \eqref{eq:vertical-trace}.
\end{proof}

\section{Separated Ricci spectral clusters}\label{sec:cluster}

The following finite-dimensional lemma gives the perturbation formula used below. All spaces are finite-dimensional real Euclidean spaces, all adjoints are Euclidean, and all remainders are measured in operator norm.

\begin{lemma}\label{lem:cluster}
Let $\mathcal H$ and $\mathcal V$ be Euclidean spaces, and let $\mathcal T(t)$ be a $C^4$ family of self-adjoint endomorphisms of $\mathcal H\oplus\mathcal V$, defined for $t$ in a neighborhood of $0$, such that
\begin{equation}
 \mathcal T(t)=
 \begin{pmatrix}
  A+t^2B+O(t^4)&tC+O(t^3)\\
 tC^*+O(t^3)&D+t^2E+O(t^4)
 \end{pmatrix}.
 \label{eq:abstract-block-family}
\end{equation}
Here $A,B\in\Sym(\mathcal H)$, $D,E\in\Sym(\mathcal V)$, and $C\in\operatorname{Hom}(\mathcal V,\mathcal H)$. Assume that
\begin{equation}
 \mathcal T(-t)=J\mathcal T(t)J,
 \qquad
 J=\diag(\Id_{\mathcal H},-\Id_{\mathcal V}),
 \label{eq:cluster-parity}
\end{equation}
and that $\spec(A)\cap\spec(D)=\varnothing$. Let $X:\mathcal V\to\mathcal H$ be the unique solution of the Sylvester equation
\begin{equation}
 AX-XD=C.
 \label{eq:sylvester}
\end{equation}
For sufficiently small $|t|$, the spectral cluster of $\mathcal T(t)$ issuing
from $\spec(D)$---the group of eigenvalues converging to $\spec(D)$,
counted with multiplicity---is well defined, has total multiplicity
$\dim\mathcal V$, and has trace
\begin{equation}
 T_D(t)=\tr D
 +t^2\left(
   \tr E-\tr(C^*X)\right)
 +O(t^4).
 \label{eq:cluster-trace-formula}
\end{equation}
If, in addition,
\begin{equation}
 \max\spec(D)<\min\spec(A),
 \label{eq:ordered-gap}
\end{equation}
then this cluster consists of the $\dim\mathcal V$ lowest eigenvalues of $\mathcal T(t)$ for all sufficiently small $|t|$.
\end{lemma}

\begin{proof}
Choose orthonormal eigenbases $Au_\alpha=a_\alpha u_\alpha$ and $Dv_\beta=d_\beta v_\beta$. Equation \eqref{eq:sylvester} has the unique solution
\[
  \langle u_\alpha,Xv_\beta\rangle
  =\frac{\langle u_\alpha,Cv_\beta\rangle}{a_\alpha-d_\beta}
\]
because the spectra are disjoint. Consequently,
\[
  \tr(C^*X)=\sum_{\alpha,\beta}
  \frac{|\langle u_\alpha,Cv_\beta\rangle|^2}{a_\alpha-d_\beta}.
\]
Under the ordered gap~\eqref{eq:ordered-gap}, every denominator is
positive. Thus mixing with the horizontal block contributes a
nonpositive term to the second-order trace of the lower cluster.

Set
\[
  \begin{aligned}
    \mathcal T_0&=\diag(A,D),\\
    \mathcal T_1&=\begin{pmatrix}0&C\\C^*&0\end{pmatrix},\\
    S&=\begin{pmatrix}0&-X\\X^*&0\end{pmatrix}.
  \end{aligned}
\]
Then $S^*=-S$ and $[\mathcal T_0,S]=-\mathcal T_1$. The Baker--Campbell--Hausdorff expansion gives
\begin{equation}
 \begin{aligned}
 \widetilde{\mathcal T}(t)&:=e^{-tS}\mathcal T(t)e^{tS}\\
 &=\mathcal T_0+t^2
 \begin{pmatrix}
 B+\frac12(CX^*+XC^*)&0\\
 0&E-\frac12(C^*X+X^*C)
 \end{pmatrix}
 +O(t^3).
 \end{aligned}
 \label{eq:block-diagonalized-family}
\end{equation}
Thus the second-order vertical block in \eqref{eq:block-diagonalized-family} is
\[
  K_{\mathcal V}=E-\frac12(C^*X+X^*C).
\]

After complexification, choose a finite union $\mathcal C$ of
positively oriented closed contours with winding number one on
$\spec(D)$ and zero on $\spec(A)$. For sufficiently small $|t|$,
define the Riesz projection~\cite[Chapter~II]{Ka1995} by
\[
  P(t)=\frac{1}{2\pi i}
  \int_{\mathcal C}(zI-\mathcal T(t))^{-1}\,dz.
\]
The operator $P(t)$ is the orthogonal projection onto the direct
sum of the eigenspaces corresponding to the spectral cluster
issuing from $\spec(D)$. It has rank $\dim\mathcal V$ and depends
$C^4$-smoothly on $t$. The sum of the eigenvalues in this cluster,
counted with multiplicity, is therefore
$
  T_D(t)=\tr\bigl(\mathcal T(t)P(t)\bigr).
$

Set
\[
  \begin{aligned}
    \widetilde P(t)&=e^{-tS}P(t)e^{tS},\\
    P_{\mathcal V}&=\diag(0,\Id_{\mathcal V}).
  \end{aligned}
\]
Equation \eqref{eq:block-diagonalized-family} gives $\widetilde{\mathcal T}'(0)=0$, so differentiating the resolvent integral gives $\widetilde P'(0)=0$. Write
\[
  \widetilde P(t)=P_{\mathcal V}+t^2P_2+O(t^3).
\]
The identity $\widetilde P(t)^2=\widetilde P(t)$ gives $P_{\mathcal V}P_2+P_2P_{\mathcal V}=P_2$, so both diagonal blocks of $P_2$ vanish and $\tr(\mathcal T_0P_2)=0$. It follows that
\[
  T_D(t)=\tr D+t^2\tr K_{\mathcal V}+O(t^3),
\]
where $\tr K_{\mathcal V}=\tr E-\tr(C^*X)$.

By \eqref{eq:cluster-parity}, $P(-t)=JP(t)J$, and hence $T_D(-t)=T_D(t)$. Since $T_D$ is $C^4$, Taylor's theorem improves the remainder to $O(t^4)$ and proves \eqref{eq:cluster-trace-formula}. Under \eqref{eq:ordered-gap}, continuity of the ordered eigenvalues identifies this cluster with the lowest $\dim\mathcal V$ eigenvalues.
\end{proof}

The conclusion of Lemma~\ref{lem:cluster} holds uniformly for smooth
families over a compact parameter space, with a common small
$t$-interval, provided its hypotheses hold fiberwise and the
unperturbed spectra have a uniform positive separation. Indeed,
finitely many local Riesz contours give uniform resolvent bounds,
also in local orthonormal trivializations of Euclidean vector bundles.
Differentiating the resolvent formula through order four then gives
a uniform $O(t^4)$ remainder.

For \eqref{eq:kk-metric}, fiber inversion $\iota(p,\theta)=(p,-\theta)$ satisfies $\iota^*g_t=g_{-t}$. Its differential sends $E_i(-t)$ to $E_i(t)$ and $U_a$ to $-U_a$, so it is represented by $J$ in the adapted orthonormal frames. Naturality and torus invariance of Ricci give \eqref{eq:cluster-parity}. The same argument applies to an even family $H_t$, in particular to $H_t=e^{t^2\phi}H$.

\section{The spherical rank-two perturbation}\label{sec:spherical}

Let $m\geq2$. View $\SSph^m$ as the unit sphere in $\RR^{m+1}$, and let $x,y,z$ be its first three coordinate functions. Set
\begin{equation}
 K=
 \begin{pmatrix}
  x&-y\\
  -y&-x
 \end{pmatrix},
 \qquad
 F=
 \begin{pmatrix}
  dy\wedge dz\\
  dz\wedge dx
 \end{pmatrix}.
 \label{eq:spherical-KF}
\end{equation}
For $u=x^2+y^2$, direct multiplication in \eqref{eq:spherical-KF} gives
\[
  \begin{aligned}
    K^2&=u\Id,\\
    \tr K&=0,\\
    \det(e^{sK})&=e^{s\tr K}=1.
  \end{aligned}
\]
Hence $H_s:=e^{sK}$ is symmetric positive definite and has unit determinant for every $s\in\RR$.

Choose the connection curvature
\[
  F_s=F-\frac32sKF.
\]
The coefficient $-3/2$ maximizes the quadratic spherical average
within the ansatz $F+asKF$; see
Remark~\ref{rem:connection-coefficient} and its verification in
Appendix~\ref{app:algebra}. The resulting $\RR^2$-valued two-form is exact. Indeed,
\begin{align}
 (KF)^1
 &=x\,dy\wedge dz-y\,dz\wedge dx
   =d(xy\,dz),
 \notag\\
 (KF)^2
 &=-y\,dy\wedge dz-x\,dz\wedge dx
   =d\left(\frac{x^2-y^2}{2}\,dz\right).
 \label{eq:KF-exact}
\end{align}
Equation \eqref{eq:KF-exact} shows that $F_s=d\alpha_s$, with $d$ acting componentwise, where
\begin{equation}
 \alpha_s^1=y\,dz-\frac{3s}{2}xy\,dz,
 \qquad
 \alpha_s^2=z\,dx-\frac{3s}{4}(x^2-y^2)\,dz.
 \label{eq:alpha-s}
\end{equation}
Thus \eqref{eq:alpha-s} defines $\alpha_s$ globally on $\SSph^m$, and the perturbation is defined on the trivial product $\SSph^m\times\TT^2$.

The parameter $s$ controls the fiber shape $H_s$ and the connection
profile $\alpha_s$, while $t$ controls the connection strength. We
first expand in $t$ at fixed $s$, then expand the resulting coefficient
in $s$. For small $s$ and $t$, set
\[
  \eta_{s,t}^a=d\theta^a+t\alpha_s^a
\]
and define the preliminary metric
\[
  \widehat g^{\,0}_{s,t}=g_{\SSph^m}+(H_s)_{ab}\eta_{s,t}^a\eta_{s,t}^b.
\]
We use the round metric in the horizontal directions and the $H_s$-orthonormal identification in the vertical directions. The Ricci endomorphism then has the block form
\begin{equation}\label{eq:preliminary-ricci-block-form}
 \Ric_{\widehat g^{\,0}_{s,t}}^\sharp
 =
 \begin{pmatrix}
  A_s+t^2B_s&tC_s\\
  tC_s^*&D_s+t^2E_s
 \end{pmatrix}.
\end{equation}
Since $\det H_s=1$, Lemma~\ref{lem:kk-blocks} gives the coefficients in \eqref{eq:preliminary-ricci-block-form} as
\begin{align}
 (A_s)_{ij}
 &=(m-1)(g_{\SSph^m})_{ij}-\tr(L_iL_j),
 \label{eq:A-s-definition}\\
 D_s
 &=e^{sK/2}(-\nabla^iL_i)e^{-sK/2},
 \label{eq:D-s-definition}\\
 C_s
 &=\frac12\Bcal(e^{sK}F_s)e^{-sK/2},
 \notag\\
 E_s
 &=\frac14e^{sK/2}\Mcal(F_s)e^{sK/2},
 \notag
\end{align}
where $L_i=\frac12e^{-sK}\nabla_i(e^{sK})$. It also gives
\[
  \tr D_s =-\nabla^i\tr L_i =-\Delta\log\sqrt{\det H_s} =0.
\]

At $s=0$,
\[
  A_0=(m-1)\Id, \qquad D_0=0.
\]
Since $m-1>0$ and the coefficients depend smoothly on $(s,p)$, some $s_0>0$ satisfies
\begin{equation}\label{eq:spherical-uniform-gap}
 \begin{aligned}
  \inf_{\substack{p\in\SSph^m\\ |s|\leq s_0}}
  \dist\bigl(\spec A_s(p),\spec D_s(p)\bigr)&>0,\\
  \max\spec D_s(p)&<\min\spec A_s(p)
  \quad(p\in\SSph^m,\ |s|\leq s_0).
 \end{aligned}
\end{equation}
For this choice of $s_0$, the Sylvester equation
\begin{equation}\label{eq:X-s-definition}
 A_sX_s-X_sD_s=C_s
\end{equation}
has a unique smooth solution $X_s$. Set
\begin{equation}\label{eq:q-s-definition}
 q_s=\tr E_s-\tr(C_s^*X_s).
\end{equation}

By Lemma~\ref{lem:cluster}, $q_s$ is the coefficient of $t^2$ in the trace of the two lowest Ricci eigenvalues of the preliminary metric. The gap \eqref{eq:spherical-uniform-gap} makes this identification uniform for small $|s|$ and $|t|$.

For reference, the main quantities in the construction are:
\begin{center}
\begin{tabular}{@{}ll@{}}
$H_s,F_s$ & fiber metric and connection curvature profile,\\
$A_s,D_s$ & horizontal and vertical blocks at $t=0$,\\
$C_s,E_s$ & linear mixed and quadratic vertical coefficients,\\
$X_s$ & solution of the Sylvester equation,\\
$q_s,c_s$ & quadratic cluster coefficient and its spherical mean.
\end{tabular}
\end{center}
We now expand the cluster coefficient. Appendix~\ref{app:algebra}
contains the contractions needed to complete the proof of the next
proposition. For tensor fields and bundle maps, $O_{C^\infty}(s^j)$ means that, for every $\ell\geq0$, all covariant derivatives through order $\ell$ are bounded by $C_\ell|s|^j$. We use the round connection on $\SSph^m$ and the fixed Euclidean fiber metric. For an integrable function $f$, write
\[
  \fintavg_{\SSph^m}f\,d\mu :=\frac{1}{\vol(\SSph^m)} \int_{\SSph^m}f\,d\mu.
\]

\begin{proposition}
\label{prop:q-expansion}
As $s\to0$,
\begin{equation}\label{eq:q-expansion}
 q_s=q_0+sq_1+s^2q_2+O_{C^\infty}(s^3),
\end{equation}
where $u=x^2+y^2$, and $v=z^2$. The coefficients are
\begin{align}
 q_0
 &=1-\frac{m+1}{4}(u+2v),
 \label{eq:q-zero}\\
 q_1
 &=-\frac{3(m+2)}8x(x^2-3y^2),
 \notag\\
 q_2
 &=\frac14\bigl((m+3)u(u+2v)-5u-v\bigr).
 \label{eq:q-two}
\end{align}
Consequently,
\begin{equation}\label{eq:q-average}
 c_s:=\fintavg_{\SSph^m}q_s\,d\mu
 =\frac{s^2}{4(m+1)}+O(s^4).
\end{equation}
In particular, $c_s>0$ for all sufficiently small nonzero $s$.
\end{proposition}

\begin{proof}
Set $\lambda=m-1$. We compute the coefficients in four steps.

\smallskip
\noindent
\emph{Step 1. Matrix and diagonal-block expansions.}
The identity $K^2=u\Id$ gives
\begin{align}
 e^{sK}
 &=\Id+sK+\frac{s^2}{2}u\Id+O_{C^\infty}(s^3),
 \notag\\
 e^{-sK/2}
 &=\Id-\frac s2K+\frac{s^2}{8}u\Id
   +O_{C^\infty}(s^3),
 \label{eq:H-minus-half-expansion}\\
 e^{sK}F_s
 &=F-\frac s2KF-s^2uF+O_{C^\infty}(s^3).
 \label{eq:HF-expansion}
\end{align}
Direct multiplication gives the logarithmic derivative
\[
  e^{-sK}\nabla_i(e^{sK}) =s\nabla_iK+\frac{s^2}{2}[\nabla_iK,K] +O_{C^\infty}(s^3).
\]
Hence
\[
  L_i =\frac s2\nabla_iK+\frac{s^2}{4}[\nabla_iK,K] +O_{C^\infty}(s^3).
\]
Substituting into \eqref{eq:A-s-definition} gives
\begin{align}
 A_s
 &=\lambda\Id+s^2A_2+O_{C^\infty}(s^3),
 \notag\\
 (A_2)_{ij}
 &=-\frac14\tr\bigl((\nabla_iK)(\nabla_jK)\bigr)
 \notag\\
 &=-\frac12\bigl((dx)_i(dx)_j+(dy)_i(dy)_j\bigr).
 \label{eq:A-s-expansion}
\end{align}

Similarly,
\begin{equation}\label{eq:D-expansion}
 D_s=sD_1+O_{C^\infty}(s^3),
 \qquad
 D_1=\frac m2K.
\end{equation}
There is no order-$s^2$ term. Each entry of $K$ is a first spherical harmonic, so $\Delta K=-mK$, and
\[
  \nabla^i[\nabla_iK,K] =[\Delta K,K]+[\nabla^iK,\nabla_iK]=0.
\]
In an orthonormal frame, the last contracted commutator equals $\sum_i[\nabla_iK,\nabla_iK]=0$. Conjugation in \eqref{eq:D-s-definition} contributes only the possible term $\frac12[K,D_1]$ at order $s^2$; it also vanishes because $D_1=(m/2)K$.

\smallskip
\noindent
\emph{Step 2. The Sylvester recursion.}
Write
\begin{align}
 C_s
 &=C_0+sC_1+s^2C_2+O_{C^\infty}(s^3),
 \label{eq:C-expansion}\\
 \tr E_s
 &=e_0+se_1+s^2e_2+O_{C^\infty}(s^3),
 \notag\\
 X_s
 &=X_0+sX_1+s^2X_2+O_{C^\infty}(s^3).
 \label{eq:X-expansion}
\end{align}
The uniform gap \eqref{eq:spherical-uniform-gap} makes the inverse Sylvester operator smooth in $(s,p)$, which gives the $O_{C^\infty}(s^3)$ remainder in \eqref{eq:X-expansion}.

Comparing powers of $s$ in \eqref{eq:X-s-definition}, using \eqref{eq:A-s-expansion} and \eqref{eq:D-expansion}, we obtain
\begin{align}
 X_0
 &=\lambda^{-1}C_0,
 \label{eq:X-zero-recursion}\\
 X_1
 &=\lambda^{-1}(C_1+X_0D_1),
 \notag\\
 X_2
 &=\lambda^{-1}(C_2-A_2X_0+X_1D_1).
 \label{eq:X-two-recursion}
\end{align}

\smallskip
\noindent
\emph{Step 3. The finite spherical contraction.}
The identities
\begin{equation}\label{eq:sphere-identities}
 \begin{gathered}
 \nabla^2x_\alpha=-x_\alpha g_{\SSph^m},
 \qquad \Delta x_\alpha=-mx_\alpha,\\
 \langle dx_\alpha,dx_\beta\rangle
 =\delta_{\alpha\beta}-x_\alpha x_\beta.
 \end{gathered}
\end{equation}
reduce the coefficient matrices in \eqref{eq:C-expansion}--\eqref{eq:X-two-recursion}, and all scalar Hilbert--Schmidt contractions, to polynomials in $x,y,z$. After lowering the base index, \eqref{eq:sphere-identities} also gives the sign check
\begin{align}
 \Bcal(f\,dx_\alpha\wedge dx_\beta)
 &={\langle df,dx_\beta\rangle}\,dx_\alpha
   -{\langle df,dx_\alpha\rangle}\,dx_\beta
 \notag\\
 &\quad +(m-1)f
   (x_\alpha\,dx_\beta-x_\beta\,dx_\alpha).
 \label{eq:B-sphere-product-rule}
\end{align}
Appendix~\ref{app:algebra} computes the expansions \eqref{eq:e-coefficients} and \eqref{eq:mixed-trace-expansion}. Substituting them into \eqref{eq:q-s-definition} gives \eqref{eq:q-zero}--\eqref{eq:q-two}.

\smallskip
\noindent
\emph{Step 4. Spherical averaging.}
The normalized moments are
\begin{align}
 \fintavg_{\SSph^m}u\,d\mu&=\frac{2}{m+1},
 &\fintavg_{\SSph^m}v\,d\mu&=\frac{1}{m+1},
 \notag\\
 \fintavg_{\SSph^m}u^2\,d\mu&=\frac{8}{(m+1)(m+3)},
 &\fintavg_{\SSph^m}uv\,d\mu&=\frac{2}{(m+1)(m+3)}.
 \label{eq:spherical-moments}
\end{align}
The first line of \eqref{eq:spherical-moments} gives $\fintavg_{\SSph^m}q_0\,d\mu=0$. The function $q_1$ also has zero mean because it is odd under $x\mapsto-x$. Using the second line, we obtain
\begin{align}
 \fintavg_{\SSph^m}u(u+2v)\,d\mu
 &=\frac{12}{(m+1)(m+3)},
 \notag\\
 \fintavg_{\SSph^m}q_2\,d\mu
 &=\frac14\left(
   \frac{12}{m+1}
   -\frac{10}{m+1}-\frac{1}{m+1}\right)
 =\frac{1}{4(m+1)}.
 \label{eq:q2-average}
\end{align}
Averaging \eqref{eq:q-expansion} first gives
\[
  c_s=\frac{s^2}{4(m+1)}+O(s^3).
\]
To improve the remainder, let $\mathfrak a$ be the antipodal map of
$\SSph^m$. From \eqref{eq:spherical-KF} and~\eqref{eq:alpha-s},
\[
  \mathfrak a^*K=-K,\qquad
  \mathfrak a^*F=F,\qquad
  \mathfrak a^*\alpha_s=\alpha_{-s}.
\]
Hence $\mathfrak a^*H_s=H_{-s}$ and
$\mathfrak a^*F_s=F_{-s}$. The map
$\mathfrak a\times\Id_{\TT^2}$ pulls
$\widehat g^{\,0}_{s,t}$ back to $\widehat g^{\,0}_{-s,t}$.
Naturality of Ricci and comparison of the quadratic cluster
coefficients give
\[
  q_s(\mathfrak a(p))=q_{-s}(p).
\]
Since the antipodal map preserves the spherical measure, $c_s=c_{-s}$.
The smooth function $c_s$ is therefore even, which improves the
remainder to $O(s^4)$ and proves~\eqref{eq:q-average}.
\end{proof}

\begin{remark}\label{rem:connection-coefficient}
Keeping $H_s=e^{sK}$ and replacing $F_s$ by $F+asKF$, let $q_2(a)$
denote the coefficient of $s^2$ in the corresponding cluster
coefficient. The calculation at the end of Appendix~\ref{app:algebra}
gives
\begin{equation}\label{eq:optimal-connection-coefficient}
  \fintavg_{\SSph^m}q_2(a)\,d\mu
  =\frac{1}{4(m+1)}
  -\frac{m(m+2)}{(m+1)(m-1)(m+3)}
   \left(a+\frac32\right)^2.
\end{equation}
For $m\geq2$, this is uniquely maximized at $a=-3/2$,
which explains the choice of $F_s$.
\end{remark}

\section{Pointwise correction and the construction theorem}\label{sec:correction}

Let $c_s$ be the spherical average in \eqref{eq:q-average}. Since $\fintavg_{\SSph^m}(q_s-c_s)\,d\mu=0$, and $\SSph^m$ is closed and connected, there is a unique mean-zero solution $\phi_s$ of
\begin{equation}\label{eq:exact-poisson}
 \Delta\phi_s=q_s-c_s,
 \qquad
 \fintavg_{\SSph^m}\phi_s\,d\mu=0.
\end{equation}
If a subscript $0$ denotes the mean-zero subspace, then for every $\ell\geq0$ and $0<\alpha<1$,
\[
  \Delta^{-1}:C^{\ell,\alpha}_0(\SSph^m) \longrightarrow C^{\ell+2,\alpha}_0(\SSph^m)
\]
is continuous, with the analogous statement in every Sobolev scale. Since $q_s-c_s$ depends smoothly on $s$ in $C^\infty(\SSph^m)$, the solution $\phi_s$ also depends smoothly on $s$ in $C^\infty(\SSph^m)$.

The first terms illustrate the correction explicitly. With
$w=x(x^2-3y^2)$, the coefficients in~\eqref{eq:q-expansion} satisfy
\[
  \Delta q_0=-2(m+1)q_0,\qquad \Delta w=-3(m+2)w.
\]
Indeed, $q_0$ and $w$ are spherical harmonics of degrees two and
three, respectively. Applying the mean-zero inverse of $\Delta$ gives
\[
  \phi_0=-\frac{q_0}{2(m+1)}=\frac{u+2v}{8}-\frac{1}{2(m+1)},
  \qquad
  \phi_s=\phi_0+\frac{s}{8}w+O_{C^\infty}(s^2).
\]
We use the exact solution~\eqref{eq:exact-poisson} in the construction.

\begin{lemma}
\label{lem:poisson-shift}
Let $(B,h)$ be closed, let $H:B\to\Sym^+(2)$ satisfy $\det H=1$, and let $\phi\in C^\infty(B)$. Suppose that the horizontal and vertical Ricci blocks at $t=0$ are spectrally separated. Keeping $h$, $\alpha$, and $\phi$ fixed, replace $H$ in \eqref{eq:kk-metric} by $\widetilde H_t=e^{t^2\phi}H$. If $q$ and $\widetilde q$ denote, respectively, the second-order coefficients in the traces of the Ricci clusters issuing from the vertical block before and after this replacement, then
\begin{equation}\label{eq:poisson-shift-conclusion}
 \widetilde q=q-\Delta_h\phi.
\end{equation}
\end{lemma}

\begin{proof}
The new fiber-volume density is $\widetilde\nu_t=e^{t^2\phi}$, so
\[
  -\widetilde\nu_t^{-1}\Delta_h\widetilde\nu_t=-t^2\Delta_h\phi-t^4|d\phi|_h^2.
\]
In \eqref{eq:vertical-trace}, the $F$-quadratic term already has a factor $t^2$, and multiplication of $H$ by $e^{t^2\phi}$ changes it only by $O(t^4)$. Thus the coefficient of $t^2$ in the vertical trace changes by $-\Delta_h\phi$.

The blocks at $t=0$ are unchanged. Moreover,
\[
  \begin{aligned}
    \widetilde\nu_t\widetilde H_t&=e^{2t^2\phi}H,\\
    \widetilde H_t^{-1/2}&=e^{-t^2\phi/2}H^{-1/2},
  \end{aligned}
\]
so \eqref{eq:C-block} gives the same first-order mixed coefficient $C$. Hence the Sylvester solution $X$ is unchanged. The horizontal block may change at order $t^2$, but this coefficient does not enter \eqref{eq:cluster-trace-formula}. The cluster coefficient therefore changes by $-\Delta_h\phi$, proving \eqref{eq:poisson-shift-conclusion}.
\end{proof}

We now define the final family. For small $|s|$ and $|t|$, set on $\SSph^m\times\TT^2$
\begin{equation}\label{eq:final-two-parameter-metric}
 \widehat g_{s,t}
 =g_{\SSph^m}
 +e^{t^2\phi_s}(e^{sK})_{ab}
 (d\theta^a+t\alpha_s^a)(d\theta^b+t\alpha_s^b),
 \qquad a,b\in\{1,2\}.
\end{equation}
For $k\geq2$, add $k-2$ flat torus directions and set
\begin{align}
 g_{s,t}^{(k)}
 &=\widehat g_{s,t}+\sum_{A=3}^k(d\theta^A)^2,
 \label{eq:final-rank-k-metric}\\
 g_s^{(k)}
 &=g_{s,s^2}^{(k)}.
 \notag
\end{align}
When $k=2$, the sum in \eqref{eq:final-rank-k-metric} is empty. We use the canonical permutation of factors to regard $g_{s,t}^{(k)}$ as a metric on $\TT^k\times\SSph^m$.

\begin{proof}[Proof of Theorem~\ref{thm:construction}]
\smallskip
\noindent
 Let $\Lambda_2(s,t,p)$ be the trace of the Ricci cluster of \eqref{eq:final-two-parameter-metric} issuing from the torus directions. The scalar correction preserves parity and the gap \eqref{eq:spherical-uniform-gap} at $t=0$. Smoothness of $H_s$, $\alpha_s$, and $\phi_s$ on the compact parameter set gives uniform bounds for the required derivatives. Lemmas~\ref{lem:poisson-shift} and~\ref{lem:cluster}, with the uniformity observation following the latter and \eqref{eq:exact-poisson}, therefore give
\[
  \begin{aligned}
    q_s(p)-\Delta\phi_s(p)&=c_s,\\
    \Lambda_2(s,t,p)&=t^2c_s+O(t^4)
  \end{aligned}
\]
uniformly for $p\in\SSph^m$, $|s|\leq s_0$, and sufficiently small $|t|$.

The positive term has size $t^2s^2$, whereas the error has size
$t^4$. Thus $t=o(s)$ makes the error smaller than the leading term.
Choose the smooth scale $t=s^2$ and use~\eqref{eq:q-average}.
Uniformly on the sphere,
\begin{align}
 \Lambda_2(s,s^2,p)
 &=s^4\left(\frac{s^2}{4(m+1)}+O(s^4)\right)+O(s^8)
 \notag\\
 &=\frac{s^6}{4(m+1)}+O(s^8).
 \label{eq:final-cluster-asymptotic}
\end{align}

 At $(s,t)=(0,0)$, the two torus eigenvalues vanish and the remaining $m$ Ricci eigenvalues equal $m-1$. The uniform ordered gap identifies the selected cluster with the two lowest Ricci eigenvalues of $\widehat g_{s,t}$ for small $(s,t)$. Denote them by $\mu_1(s,t,p)$ and $\mu_2(s,t,p)$, counted with multiplicity.

The metric $g_{s,t}^{(k)}$ is the product of $\widehat g_{s,t}$ and a flat $(k-2)$-torus. Hence, as a multiset,
\begin{equation}\label{eq:product-spectrum}
 \spec\bigl(\Ric_{g_{s,t}^{(k)}}^\sharp\bigr)
 =\spec\bigl(\Ric_{\widehat g_{s,t}}^\sharp\bigr)
 \cup\{\underbrace{0,\ldots,0}_{k-2\text{ times}}\}.
\end{equation}
The active eigenvalues converge uniformly to zero, while every other eigenvalue of $\widehat g_{s,t}$ converges uniformly to $m-1>0$. Thus the $k$ lowest eigenvalues in \eqref{eq:product-spectrum} are $\mu_1,\mu_2$ and the $k-2$ zeros, regardless of their internal order. Equation \eqref{eq:final-cluster-asymptotic} proves \eqref{eq:construction-asymptotic} uniformly on the full product by torus invariance. Its positive leading coefficient gives $\RicSum{k}(g_s^{(k)})>0$ for all sufficiently small $s>0$.

The forms $d\theta^A$ are global invariant one-forms on the torus, and $x,y,z$, $\alpha_s$, and $\phi_s$ are global smooth data on the sphere. The matrix $e^{t^2\phi_s}e^{sK}$ is a positive scalar multiple of a positive-definite symmetric matrix. Thus \eqref{eq:final-rank-k-metric} defines a global smooth Riemannian metric. When $t=s^2$,
\[
  \begin{aligned}
    e^{sK}&=\Id+O_{C^\infty}(s),\\
    t\alpha_s&=O_{C^\infty}(s^2),\\
    t^2\phi_s&=O_{C^\infty}(s^4).
  \end{aligned}
\]
Therefore the family extends smoothly to $s=0$ and converges in $C^\infty$ to the standard flat--round product $g_0^{(k)}$.
\end{proof}

\section{A local Bochner obstruction}\label{sec:obstruction}

We prove Theorem~\ref{thm:obstruction-intro}. Let $\eta^1,\eta^2,\eta^3$ be the parallel orthonormal one-forms in its hypotheses. They are closed. For $g$ close to $g_0$, let $\omega_g^a$ be the $g$-harmonic representative of $[\eta^a]$. Since $\omega_g^a$ and $\eta^a$ represent the same cohomology class, there is a unique function $u_g^a$, normalized to have zero $g_0$-mean, such that
\[
  \omega_g^a=\eta^a+du_g^a, \qquad \int_Mu_g^a\,d\mu_{g_0}=0.
\]
If $\delta_g$ denotes the $L^2(g)$-adjoint of $d$, then $u_g^a$ solves
\begin{equation}\label{eq:harmonic-representative-equation}
 \delta_gdu_g^a=-\delta_g\eta^a.
\end{equation}
\begin{lemma}
\label{lem:harmonic-frame}
If $g\to g_0$ in $C^{2,\alpha}$, then
\[
  \begin{aligned}
    u_g^a&\longrightarrow0 \quad\text{in }C^{3,\alpha},\\
    X_a:=(\omega_g^a)^{\sharp_g}&\longrightarrow(\eta^a)^{\sharp_{g_0}}\quad\text{uniformly}.
  \end{aligned}
\]
Consequently, $X_1,X_2,X_3$ are pointwise linearly independent throughout a sufficiently small $C^{2,\alpha}$-neighborhood of $g_0$.
\end{lemma}

\begin{proof}
Use the fixed complements
\[
  \begin{aligned}
    \mathcal E&=\{u\in C^{3,\alpha}(M):\int_Mu\,d\mu_{g_0}=0\},\\
    \mathcal F&=\{f\in C^{1,\alpha}(M):\int_Mf\,d\mu_{g_0}=0\}.
  \end{aligned}
\]
Let $\Pi_0:C^{1,\alpha}(M)\to\mathcal F$ subtract the $g_0$-mean. Define
\[
  \widehat L_g :=\Pi_0\circ\delta_gd: \mathcal E\longrightarrow\mathcal F.
\]
Because $M$ is closed and connected, $\delta_{g_0}d$ has the constants as its kernel and the $g_0$-mean-zero functions as its range. Global Schauder theory therefore shows that $\widehat L_{g_0}:\mathcal E\to\mathcal F$ is an isomorphism. The coefficients of $\delta_gd$ depend continuously on $g$ in the norms for maps $C^{3,\alpha}\to C^{1,\alpha}$, so
\[
  \widehat L_g\longrightarrow\widehat L_{g_0} \quad\text{in } \mathcal L(\mathcal E,\mathcal F).
\]
After shrinking the neighborhood, the Neumann perturbation argument shows that every $\widehat L_g$ is invertible and that
\begin{equation}\label{eq:uniform-projected-inverse}
 \sup_g\|\widehat L_g^{-1}\|_{\mathcal L(\mathcal F,\mathcal E)}
 <\infty.
\end{equation}

Put $f_g^a=-\delta_g\eta^a$ and define $u_g^a=\widehat L_g^{-1}\Pi_0f_g^a$. Then $\Pi_0(\delta_gdu_g^a-f_g^a)=0$, so $\delta_gdu_g^a-f_g^a$ is constant. Both terms have zero $g$-mean:
\[
  \int_M\delta_gdu_g^a\,d\mu_g=0, \qquad \int_Mf_g^a\,d\mu_g=0.
\]
The constant is therefore zero, so $u_g^a$ solves \eqref{eq:harmonic-representative-equation}. By uniqueness, this is the normalized potential defined above. Equation \eqref{eq:uniform-projected-inverse} gives the uniform estimate
\begin{equation}\label{eq:uniform-schauder}
 \|u_g^a\|_{C^{3,\alpha}(g_0)}
 \leq C\|\delta_g\eta^a\|_{C^{1,\alpha}(g_0)}.
\end{equation}

Since $\eta^a$ is parallel, $\delta_{g_0}\eta^a=0$, while $g\to g_0$ in $C^{2,\alpha}$ gives
\[
  \delta_g\eta^a\longrightarrow0 \quad\text{in }C^{1,\alpha}.
\]
Estimate \eqref{eq:uniform-schauder} now gives $u_g^a\to0$ in $C^{3,\alpha}$. Raising an index is continuous in $g$, so $X_a\to(\eta^a)^{\sharp_{g_0}}$ uniformly. The limiting fields form a pointwise orthonormal $3$-frame, and their independence therefore persists uniformly.
\end{proof}

\begin{lemma}
\label{lem:three-dimensional-trace}
Let $Q$ be self-adjoint on a three-dimensional Euclidean space, with eigenvalues $\rho_1\leq\rho_2\leq\rho_3$, and put $|Q|=(Q^2)^{1/2}$. If $\rho_1+\rho_2>0$, then
\begin{equation}\label{eq:trace-norm-control}
 \tau:=\tr Q>0,
 \qquad
 \tr|Q|<3\tau.
\end{equation}
\end{lemma}

\begin{proof}
If $\rho_1\geq0$, then the hypothesis gives $\tau>0$, and hence $\tr|Q|=\tau<3\tau$. Otherwise put $a=|\rho_1|$. Then $a<\rho_2\leq\rho_3$, so $\tau=-a+\rho_2+\rho_3>a>0$. Therefore
\[
  \tr|Q|=a+\rho_2+\rho_3=\tau+2a<3\tau.\qedhere
\]

\end{proof}

\begin{remark}
The corresponding trace estimate fails uniformly in dimension two.
For $\varepsilon>0$, let $Q_\varepsilon=\diag(-1,1+\varepsilon)$.
Then
\[
  \tr Q_\varepsilon=\varepsilon>0,
  \qquad
  \frac{\tr|Q_\varepsilon|}{\tr Q_\varepsilon}
  =\frac{2+\varepsilon}{\varepsilon}\longrightarrow\infty
  \quad\text{as }\varepsilon\downarrow0.
\]
Thus the argument below depends on the third harmonic direction.
\end{remark}

\begin{proof}[Proof of Theorem~\ref{thm:obstruction-intro}]
Set $P=\operatorname{span}\{X_1,X_2,X_3\}\subset TM$ and take all norms and traces fiberwise with respect to $g|_P$. Define the frame operator $S:P\to P$ by
\[
  S(v)=\sum_{a=1}^3g(X_a,v)X_a.
\]
After shrinking the neighborhood, Lemma~\ref{lem:harmonic-frame} gives
\begin{equation}\label{eq:S-close}
 \|S-\Id_P\|_{\mathrm{op}}<\frac13
\end{equation}
pointwise, since the Gram matrix $(g(X_a,X_b))_{a,b=1}^3$ converges uniformly to the identity and has the same eigenvalues as $S$.

Let $Q=\pi_P\circ\Ric_g^\sharp|_P$ be the self-adjoint compression to $P$, where $\pi_P$ is the orthogonal projection. Thus $g(Qu,v)=\Ric_g(u,v)$ for $u,v\in P$; no invariance of $P$ is required. If $\rho_1\leq\rho_2\leq\rho_3$ are its eigenvalues, Ky Fan's principle \eqref{eq:ky-fan} gives $\rho_1+\rho_2\geq\RicSum{2}(g,p)$. Assuming $\RicSum{2}(g)>0$ everywhere, Lemma~\ref{lem:three-dimensional-trace} applies to $Q$ pointwise.

In an orthonormal eigenbasis of $Q$, we obtain
\begin{equation}
 \left|\tr\bigl((S-\Id_P)Q\bigr)\right|
 \leq\|S-\Id_P\|_{\mathrm{op}}\tr|Q|.
 \label{eq:trace-duality-compressed}
\end{equation}
This estimate does not require $S$ and $Q$ to commute. Since $S=\sum_aX_a\otimes X_a^{\flat_g}$, we have $\sum_a\Ric_g(X_a,X_a)=\tr(SQ)$. Equations \eqref{eq:S-close}, \eqref{eq:trace-norm-control}, and \eqref{eq:trace-duality-compressed} therefore give
\begin{align}
 \sum_{a=1}^3\Ric_g(X_a,X_a)
 &=\tr(SQ)
 \notag\\
 &\geq\tr Q
   -\|S-\Id_P\|_{\mathrm{op}}\tr|Q|>0.
 \label{eq:positive-compressed-trace}
\end{align}
Since $\omega_g^a$ is harmonic and $M$ is closed, the integrated Bochner formula gives
\begin{equation}\label{eq:bochner-identity}
 \int_M\Ric_g(X_a,X_a)\,d\mu_g
 =-\int_M|\nabla\omega_g^a|^2\,d\mu_g.
\end{equation}
Summing \eqref{eq:bochner-identity} over $a=1,2,3$ shows that the integral of $\sum_a\Ric_g(X_a,X_a)$ is nonpositive. This contradicts \eqref{eq:positive-compressed-trace}, because that function is continuous and strictly positive on $M$.
\end{proof}

\begin{corollary}\label{cor:torus-obstruction}
Let $k\geq3$, let $(Y,g_Y)$ be a closed connected Riemannian manifold, and let $g_{\TT^k}$ be any flat metric on $\TT^k$. For every $0<\alpha<1$, there is a $C^{2,\alpha}$-neighborhood of $g_{\TT^k}+g_Y$ on $\TT^k\times Y$ that contains no smooth metric with $\RicSum{2}>0$ everywhere.
\end{corollary}

\begin{proof}
Since $b_1(\TT^k)=k$ and $\Ric_{g_{\TT^k}}=0$, the Bochner formula implies that every harmonic one-form is parallel. Evaluation at one point is injective on parallel one-forms. Since the harmonic space and each cotangent space both have dimension $k$, evaluation is an isomorphism. Orthonormalizing a harmonic basis at one point gives a global parallel orthonormal coframe. The pullbacks of any three members are parallel and orthonormal for the product metric on $\TT^k\times Y$, so Theorem~\ref{thm:obstruction-intro} applies.
\end{proof}

\appendix

\section{Exact contraction table}\label{app:algebra}

This appendix records the finite-dimensional calculation used in Proposition~\ref{prop:q-expansion}. Set
\begin{equation}
 (x_1,x_2,x_3)=(x,y,z),
 \qquad
 \xi=\begin{pmatrix}x\\y\\z\end{pmatrix},
 \qquad
 \Gamma=\Id_3-\xi\xi^T,
 \qquad
 P_0=\diag(1,1,0).
 \label{eq:appendix-data}
\end{equation}
The matrix $\Gamma$ in \eqref{eq:appendix-data} is the Gram matrix of the first three coordinate one-forms. Its entries are $\Gamma_{\mu\nu} =\langle dx_\mu,dx_\nu\rangle$ for $1\leq\mu,\nu\leq3$.

Let $\varepsilon_1,\varepsilon_2$ be the standard basis of $\RR^2$. A matrix $Y\in\RR^{3\times2}$ represents the bundle map
\[
  \begin{gathered}
    \Phi(Y):\RR^2\longrightarrow T\SSph^m,\\
    \Phi(Y)\varepsilon_a =\sum_{\mu=1}^3Y_{\mu a}\nabla x_\mu.
  \end{gathered}
\]
For coefficient matrices $U,V$, define
\begin{equation}\label{eq:Gamma-pairing}
 \langle U,V\rangle_\Gamma
 :=\tr(U^T\Gamma V).
\end{equation}
The pairing \eqref{eq:Gamma-pairing} is the intrinsic Hilbert--Schmidt pairing of the represented maps, even where the coefficient representation is nonunique; this includes every point when $m=2$.

Raise the base index of $\Bcal(G)$ with the round metric before using this $3\times2$ representation. Define
\begin{align}
 M_0&=
 \begin{pmatrix}
 0&z\\
 -z&0\\
 y&-x
 \end{pmatrix},
 &
 M_1&=
 \begin{pmatrix}
 yz&xz\\
 xz&-yz\\
 -2xy&y^2-x^2
 \end{pmatrix},
 \notag\\
 M_2&=uM_0,
 &
 N_0&=
 \begin{pmatrix}
 0&0\\
 0&0\\
 -y&x
 \end{pmatrix},
 \notag\\
 M_3&=
 \begin{pmatrix}
 0&z/4\\
 -z/4&0\\
 y&-x
 \end{pmatrix}.
 \label{eq:appendix-matrices}
\end{align}

Applying \eqref{eq:B-sphere-product-rule} to $F$, $KF$, and $K^2F=uF$ gives the following identities in the coefficient matrices \eqref{eq:appendix-matrices}
\begin{equation}\label{eq:B-identities}
\begin{aligned}
 \Bcal(F)&=(m-1)M_0,\\
 \Bcal(KF)&=-mM_1,\\
 \Bcal(K^2F)&=(m+1)M_2+2N_0.
 \end{aligned}
\end{equation}

For example, before raising the base index, the product
rule~\eqref{eq:B-sphere-product-rule} gives
\[
  \begin{aligned}
    \Bcal(x\,dy\wedge dz)&=mx(y\,dz-z\,dy),\\
    \Bcal(-y\,dz\wedge dx)&=my(x\,dz-z\,dx).
  \end{aligned}
\]
Adding these identities yields
\[
  \Bcal((KF)^1)=-m(yz\,dx+xz\,dy-2xy\,dz).
\]
Similarly,
\[
  \Bcal((KF)^2)=-m\bigl(xz\,dx-yz\,dy+(y^2-x^2)\,dz\bigr).
\]
These are precisely the two columns of $-mM_1$ after raising the
base index.

We next compute the Sylvester coefficients. We use the matrix identities
\begin{equation}\label{eq:basic-matrix-identities}
 \begin{gathered}
 M_0K=-M_1,\qquad M_1K=-M_2,\\
 \Gamma M_0=M_0,\qquad \Gamma M_1=M_1,\qquad
 \frac14P_0\Gamma M_0=N_0+M_3.
 \end{gathered}
\end{equation}
The $\Gamma$ identities follow from $\xi^TM_0=\xi^TM_1=0$. Direct multiplication by $P_0$ gives the final equality.

The Sylvester recursion uses
\[
  A_2=-\frac12(dx\otimes dx+dy\otimes dy), \qquad D_1=\frac m2K.
\]
Raise one index of $A_2$ with the round metric. If $Y$ represents a map, then the following matrix represents $A_2Y$:
\begin{equation}\label{eq:A2-coefficient-action}
 -\frac12P_0\Gamma Y.
\end{equation}

Using \eqref{eq:HF-expansion}, \eqref{eq:H-minus-half-expansion}, \eqref{eq:B-identities}, and \eqref{eq:basic-matrix-identities} in the definition of $C_s$, then applying the recursion \eqref{eq:X-zero-recursion}--\eqref{eq:X-two-recursion}, we obtain
\begin{align}
 C_0&=\frac{m-1}{2}M_0,
 &X_0&=\frac12M_0,
 \label{eq:C0-X0}\\
 C_1&=\frac{2m-1}{4}M_1,
 &X_1&=\frac14M_1,
 \notag\\
 C_2&=-\frac{5m+9}{16}M_2-N_0,
 &X_2&=-\frac{7m+9}{16(m-1)}M_2
       +\frac1{m-1}M_3.
 \label{eq:C2-X2}
\end{align}
For example, \eqref{eq:A2-coefficient-action} and \eqref{eq:basic-matrix-identities} give
\[
  A_2X_0=-\frac14P_0\Gamma M_0=-(N_0+M_3),
\]
which gives the term in $X_2$ arising from the horizontal correction $A_2$. Since $m\geq2$, every denominator $m-1$ is nonzero.

We finally compute the vertical trace and the required contractions. Put
\[
  \rho=u^2+2uv=u(u+2v), \qquad w=x(x^2-3y^2).
\]
The two-form convention \eqref{eq:two-form-convention} gives
\[
  \tr E_s =\frac12\langle e^{sK}F_s,F_s\rangle.
\]
The required two-form contractions are
\[
  \begin{aligned}
    |F|^2&=2-u-2v,\\
    \langle F,KF\rangle&=w,\\
    |KF|^2&=u|F|^2.
  \end{aligned}
\]
Expanding $\tr E_s$ gives
\begin{equation}\label{eq:e-coefficients}
 \tr E_s=e_0+se_1+s^2e_2+O_{C^\infty}(s^3),
\end{equation}
where
\begin{equation*}
   e_0=1-\frac u2-v,\quad
 e_1=-w,\quad
 e_2=\frac{\rho-2u}{8}.  
\end{equation*}

For maps in the coefficient representation,
\[
  \tr(C_j^*X_\ell)=\langle C_j,X_\ell\rangle_\Gamma.
\]
The six needed contractions reduce to
\begin{align}
 \langle M_0,M_0\rangle_\Gamma&=u+2v,
 &
 \langle M_1,M_0\rangle_\Gamma&=w,
 \notag\\
 \langle M_2,M_0\rangle_\Gamma&=\rho,
 &
 \langle N_0,M_0\rangle_\Gamma&=-u,
 \notag\\
 \langle M_1,M_1\rangle_\Gamma&=\rho,
 &
 \langle M_0,M_3\rangle_\Gamma&=u+\frac v2.
 \label{eq:basic-contractions}
\end{align}
For instance, the contribution from $C_0^*X_2$ to the
coefficient of $s^2$ is
\[
  \begin{aligned}
    \tr(C_0^*X_2)
    &=-\frac{7m+9}{32}\langle M_0,M_2\rangle_\Gamma
      +\frac12\langle M_0,M_3\rangle_\Gamma\\
    &=-\frac{7m+9}{32}\rho+\frac u2+\frac v4.
  \end{aligned}
\]
The full quadratic coefficient is
\[
  \tr(C_2^*X_0)+\tr(C_1^*X_1)+\tr(C_0^*X_2).
\]
Using~\eqref{eq:C0-X0}--\eqref{eq:C2-X2}
and~\eqref{eq:basic-contractions}, we obtain
\begin{align}
 \tr(C_s^*X_s)
 &=\frac{m-1}{4}(u+2v)+\frac{3m-2}{8}sw
 \notag\\
 &\quad+s^2\left(-\frac{2m+5}{8}\rho+u+\frac v4\right)
 +O_{C^\infty}(s^3).
 \label{eq:mixed-trace-expansion}
\end{align}

\subsection*{The choice of connection coefficient}
We verify~\eqref{eq:optimal-connection-coefficient}. Keep $H_s=e^{sK}$
and use $F_s=F+asKF$ for a constant $a\in\RR$. The same expansions
and divergence identities give
\[
  \begin{aligned}
    C_1(a)&=-\frac{2ma+m+1}{4}M_1,\\
    C_2(a)&=\left(\frac{a(m+2)}4+\frac{m+3}{16}\right)M_2
          +\left(a+\frac12\right)N_0,\\
    X_1(a)&=-\frac{2ma+2m+1}{4(m-1)}M_1.
  \end{aligned}
\]
The vertical quadratic coefficient is
\[
  e_2(a)=\left(a^2+2a+\frac12\right)
             \left(u-\frac\rho2\right).
\]
Substituting these expressions into the Sylvester
recursion~\eqref{eq:X-two-recursion} and using the contraction
table~\eqref{eq:basic-contractions} yields
\[
  \begin{aligned}
    q_2(a)&=(a^2+3a+1)u-\frac v4\\
    &\quad-\left(
      \frac{a^2}{2}+\frac{m+6}{4}a+\frac{m+7}{16}
      +\frac{(2ma+2m+1)^2}{16(m-1)}
      \right)\rho.
  \end{aligned}
\]
Averaging with~\eqref{eq:spherical-moments} and completing the
square proves~\eqref{eq:optimal-connection-coefficient}.

\end{document}